\documentclass[11pt]{amsart}

\author[C.~Sanna]{Carlo Sanna$^\dagger$}
\address{\parbox{\linewidth}{
		Department of Mathematical Sciences, Politecnico di Torino\\
		Corso Duca degli Abruzzi 24, 10129 Torino, Italy\\[-8pt]}}
\email{carlo.sanna@polito.it}
\thanks{$\dagger$ C.~Sanna is a member of the INdAM group GNSAGA and of CrypTO, the group of Cryptography and Number Theory of the Politecnico di Torino}


\title{On the index of appearance of a Lucas sequence}

\usepackage{amsmath}
\usepackage{amssymb}
\usepackage{amsthm}
\usepackage{booktabs}
\usepackage{placeins}
\usepackage{geometry}
\usepackage{bm}
\usepackage{color}
\usepackage{hyperref}
\hypersetup{colorlinks=true}
\usepackage{enumitem}
\setlist[enumerate]{label=(\roman*),labelindent=1em,itemsep=0.5em,topsep=0.5em}

\newtheorem{theorem}{Theorem}[section]

\newtheorem{lemma}[theorem]{Lemma}
\theoremstyle{remark}

\newtheorem{example}{Example}[section]

\DeclareMathOperator{\Gal}{Gal}

\DeclareMathOperator{\Norm}{N}
\DeclareMathOperator{\id}{id}

\begin{document}

\begin{abstract}
    Let $\bm{u} = (u_n)_{n \geq 0}$ be a Lucas sequence, that is, a sequence of integers satisfying $u_0 = 0$, $u_1 = 1$, and $u_n = a_1 u_{n - 1} + a_2 u_{n - 2}$ for every integer $n \geq 2$, where $a_1$ and $a_2$ are fixed nonzero integers.
    For each prime number $p$ with $p \nmid 2a_2D_{\bm{u}}$, where $D_{\bm{u}} := a_1^2 + 4a_2$, let $\rho_{\bm{u}}(p)$ be the \emph{rank of appearance} of $p$ in $\bm{u}$, that is, the smallest positive integer $k$ such that $p \mid u_k$.
    It is well known that $\rho_{\bm{u}}(p)$ exists and that $p \equiv \big(D_{\bm{u}} \mid p \big) \pmod {\rho_{\bm{u}}(p)}$, where $\big(D_{\bm{u}} \mid p \big)$ is the Legendre symbol.
    Define the \emph{index of appearance} of $p$ in $\bm{u}$ as $\iota_{\bm{u}}(p) := \left(p - \big(D_{\bm{u}} \mid p \big)\right) / \rho_{\bm{u}}(p)$.
    For each positive integer $t$ and for every $x > 0$, let $\mathcal{P}_{\bm{u}}(t, x)$ be the set of prime numbers $p$ such that $p \leq x$, $p \nmid 2a_2 D_{\bm{u}}$, and $\iota_{\bm{u}}(p) = t$.

    Under the Generalized Riemann Hypothesis, and under some mild assumptions on $\bm{u}$, we prove that
	\begin{equation*}
        \#\mathcal{P}_{\bm{u}}(t, x) = A\, F_{\bm{u}}(t) \, G_{\bm{u}}(t) \, \frac{x}{\log x} + O_{\bm{u}}\!\left(\frac{x}{(\log x)^2} + \frac{x \log (2\log x)}{\varphi(t) (\log x)^2}\right) ,
    \end{equation*}
    for all positive integers $t$ and for all $x > t^3$, where $A$ is the Artin constant, $F_{\bm{u}}(\cdot)$ is a multiplicative function, and $G_{\bm{u}}(\cdot)$ is a periodic function (both these functions are effectively computable in terms of $\bm{u}$).
    Furthermore, we provide some explicit examples and numerical data.
\end{abstract}

\maketitle

\section{Introduction}

Let $\bm{u} = (u_n)_{n \geq 0}$ be a Lucas sequence, that is, a sequence of integers satisfying $u_0 = 0$, $u_1 = 1$, and $u_n = a_1 u_{n - 1} + a_2 u_{n - 2}$ for every integer $n \geq 2$, where $a_1$ and $a_2$ are fixed nonzero integers.
For each prime number $p$ with $p \nmid a_2$, let $\rho_{\bm{u}}(p)$ be the \emph{rank of appearance} of $p$ in $\bm{u}$, that is, the smallest positive integer $k$ such that $p \mid u_k$.
It is well known that $\rho_{\bm{u}}(p)$ exists (for this and other elementary facts on $\rho_{\bm{u}}(p)$ mentioned in the introduction, see~\cite[Chapter~1]{MR1761897}).
Furthermore, we have that $p \equiv \big(D_{\bm{u}} \mid p\big) \pmod {\rho_{\bm{u}}(p)}$ for every prime number $p$ with $p \nmid 2 a_2 D_{\bm{u}}$, where $D_{\bm{u}} := a_1^2 + 4a_2$ and $\big(D_{\bm{u}} \mid p\big)$ is the Legendre symbol.
We define the \emph{index of appearance} of $p$ in $\bm{u}$ as $\iota_{\bm{u}}(p) := \left(p - \big(D_{\bm{u}} \mid p \big)\right) / \rho_{\bm{u}}(p)$, which, by the previous consideration, is a positive integer.
Moreover, for each positive integer $t$ and for every $x > 0$, we define
\begin{equation*}
	\mathcal{P}_{\bm{u}}(t, x) := \big\{p \leq x : p \nmid 2 a_2 D_{\bm{u}},\, \iota_{\bm{u}}(p) = t \big\} .
\end{equation*}
Note that $\rho_{\bm{u}}(p)$, respectively $\iota_{\bm{u}}(p)$, is somehow analog to the multiplicative order $r_g(p)$ modulo $p$, respectively the residual index $i_g(p) := (p - 1) / r_g(p)$ modulo $p$, of a rational number $g$ (assuming that $p$ does not divide the numerator and the denominator of $g$).
In particular, the set $\mathcal{P}_{\bm{u}}(t, x)$ is analog to the set of prime numbers $p$ for which $g$ is a \emph{near-primitive root} modulo $p$, that is, $i_p(g) = t$ for a fixed $t$.
This latter set has been studied by several authors \cite{MR3286520,MR1767657,MR2035537,MR3086966} (see also \cite{MR4231532,MR2314267} for generalizations to number fields).

As a first result, we provide a (conditional) asymptotic formula for $\#\mathcal{P}_{\bm{u}}(t, x)$.
We remark that the proof is a close adaptation of Hooley's proof of the Artin's conjecture under the Generalized Riemann Hypothesis (GRH)~\cite{MR207630}.

\begin{theorem}\label{thm:asymptotic}
	Assume the GRH.
	Let $\bm{u} = (u_n)_{n \geq 0}$ be a nondegenerate Lucas sequence with nonsquare $D_{\bm{u}}$.
	Then, for every positive integer $t$, there exists $\delta_{\bm{u}}(t) \geq 0$ such that
	\begin{equation*}
		\#\mathcal{P}_{\bm{u}}(t, x) = \delta_{\bm{u}}(t) \, \frac{x}{\log x} + O_{\bm{u}}\!\left(\frac{x}{(\log x)^2} + \frac{x \log (2\log x)}{\varphi(t) (\log x)^2}\right) ,
	\end{equation*}
	for all $x > t^3$.
\end{theorem}

The proof of Theorem~\ref{thm:asymptotic} yields an expression for $\delta_{\bm{u}}(t)$ in terms of an infinite series, which however is not very enlightening.
Let
\begin{equation*}
    A := \prod_p \left(1 - \frac1{p(p - 1)}\right) = 0.3739558136\dots
\end{equation*}
be the \emph{Artin constant}, and let
\begin{equation*}
	F_h(t) := \frac{(t, h)}{\varphi(t)t} \prod_{\substack{p \,>\, 2 \\ p \,\mid\, t}} \left(1 - \frac{(pt, h)}{p^2(t, h)}\right) \left(1 - \frac{(p, h)}{p(p-1)}\right)^{-1} ,
\end{equation*}
for all positive integers $h$ and $t$.
Note that $F_h(\cdot)$ is a multiplicative function.
Our second result is a more explicit description of $\delta_{\bm{u}}(t)$.

\begin{theorem}\label{thm:constant}
    Let $\bm{u} = (u_n)_{n \geq 0}$ be a nondegenerate Lucas sequence such that $D_{\bm{u}}$ is not a square and the discriminant of $\mathbb{Q}(\!  \sqrt{D_{\bm{u}}})$ is not equal to $-3$ or to $-4$.
    Then there exist a positive integer $h$ and a periodic function with rational values $G_{\bm{u}}(\cdot)$ such that
    $\delta_{\bm{u}}(t) = A \, F_{2h}(t) \, G_{\bm{u}}(t)$
    for every positive integer $t$.
\end{theorem}

The function $G_{\bm{u}}(\cdot)$ can be effectively computed in terms of $\bm{u}$.
We provide the following examples (see Section~\ref{sec:examples} for more details).

\begin{example}\label{exe:11}
	If $a_1 = a_2 = 1$, that is, if $\bm{u}$ is the sequence of Fibonacci numbers, then we have
	\begin{equation*}
		G_{\bm{u}}(t) =
			\begin{cases}
				3/4 & \text{ if } t \equiv 0 \pmod {20}; \\
				1 & \text{ if } t \equiv 1 \pmod 2; \\
				29/38 & \text{ if } t \equiv 2,6,14,18 \pmod {20}; \\
				27/76 & \text{ if } t \equiv 4,8,12,16 \pmod {20}; \\
				1/2 & \text{ if } t \equiv 10 \pmod {20}; \\
			\end{cases}
	\end{equation*}
	for all positive integers $t$.
\end{example}

\begin{example}\label{exe:4-1}
	If $a_1 = 4$ and $a_2 = -1$, then we have
    \begin{equation*}
        G_{\bm{u}}(t) =
            \begin{cases}
                3/2 & \text{ if } t \equiv 0, 2, 6, 10, 14, 18, 22 \pmod {24}; \\
                0 & \text{ if } t \equiv 1 \pmod {2}; \\
                4/5 & \text{ if } t \equiv 4, 20 \pmod {24}; \\
                3/5 & \text{ if } t \equiv 8, 16 \pmod {24}; \\
                1/2 & \text{ if } t \equiv 12 \pmod {24};
            \end{cases}
        \end{equation*}
    for all positive integers $t$.
\end{example}

\begin{example}\label{exe:102}
    If $a_1 = 10$ and $a_2 = 2$, then we have
    \begin{equation*}
        G_{\bm{u}}(t) =
        \begin{cases}
            9/10 & \text{ if } t \equiv 0\pmod {24}; \\
            3/5 & \text{ if } t \equiv 1,3,4,5,7,9,11,13,15,17,19,20,21,23\pmod {24}; \\
            9/20 & \text{ if } t \equiv 2, 6, 10, 14, 18, 22 \pmod {24}; \\
            0 & \text{ if } t \equiv 8, 16 \pmod {24}; \\
            3/10 & \text{ if } t \equiv 12 \pmod {24};
        \end{cases}
    \end{equation*}
    for all positive integers $t$.
\end{example}

For Examples~\ref{exe:11}, \ref{exe:4-1}, and~\ref{exe:102}, we compared the value of $\delta_{\bm{u}}(t)$ given by Theorem~\ref{thm:constant} with the empirical value $\widetilde{\delta}_{\bm{u}}(t) := \mathcal{P}_{\bm{u}}(t, p_{10^6}) / 10^6$, where $p_n$ denotes the $n$th prime number.
The results are in agreement, see Tables~\ref{tab:1-1}, \ref{tab:4--1}, and~\ref{tab:10-2}.

\section{Notation}

We reserve the letters $p$ and $q$ for prime numbers.
We employ the Landau--Bachmann ``Big Oh'' notation $O$, as well as the associated Vinogradov symbols $\ll$ and $\gg$, with their usual meanings.
Any dependence of the implied constants is explicitly stated or indicated with subscripts.
In particular, notations like $O_{\bm{u}}$ and $\ll_{\bm{u}}$ are shortcuts for $O_{a_1, a_2}$ and $\ll_{a_1, a_2}$, respectively.
We let $\mathbf{i}$ denote the imaginary unity, and we put $\zeta_n := \mathrm{e}^{2 \pi \mathbf{i} / n}$ for every positive integer $n$.
For every field $F$, we let $F^n := \{a^n : a \in F\}$.
If $F$ is a number field, we write $\Delta_F$ for the discriminant of $F$ over $\mathbb{Q}$, we let $\Norm_F(a)$ be the norm of $a \in K$ over $\mathbb{Q}$, and we write $a^{1/n}$ for an arbitrary, but fixed, $n$th root of $a$ (in some extension of $F$).
Moreover, $\varphi(n)$, $\mu(n)$, and $\nu_p(n)$ denote the Euler totient function, the M\"obius function, and the $p$-adic valuation, of a positive integer $n$.
Given a Galois extension $E/F$ of number fields, a prime ideal $\mathfrak{p}$ of $\mathcal{O}_F$ that does not ramify in $E$, and a prime ideal $\mathfrak{P}$ of $\mathcal{O}_E$ lying over $\mathfrak{p}$, we write $\big[E/F \mid \mathfrak{P} \big]$ for the Frobenius automorphism corresponding to $\mathfrak{P} / \mathfrak{p}$, that is, the unique $\sigma \in \Gal(E/F)$ such that $\sigma(a) \equiv a^{\Norm_F(\mathfrak{p})} \pmod {\mathfrak{P}}$ for every $a \in \mathcal{O}_E$, where $\Norm_F(\mathfrak{p})$ denotes the norm of $\mathfrak{p}$ in $F$.
Moreover, we let $\big[E/F \mid \mathfrak{p} \big]$ be the set of all $\big[E/F \mid \mathfrak{P} \big]$, where $\mathfrak{P}$ runs over the prime ideals of $\mathcal{O}_E$ lying over $\mathfrak{p}$.

% % % % % % % % % % % % % % % % % % % % % % % % % % % % % % % % % % % % % % % % %

\section{Proof of Theorem~\ref{thm:asymptotic}}

\subsection{Preliminaries}

We need the following conditional version of the Chebotarev density theorem.

\begin{theorem}\label{thm:chebotarev}
    Assume the GRH.
    Let $F / \mathbb{Q}$ be a finite Galois extension with Galois group $G$, and let $C$ be a union of conjugacy classes of $G$.
    Then we have that
    \begin{align*}
        \pi_{F, C}(x) &:= \#\big\{p \leq x : p \text{ does not ramify in } F \text{ and } \big[F/\mathbb{Q} \mid p \big] \subseteq C \big\} \\[5pt]
        &\;= \frac{\#C}{\#G} \frac{x}{\log x} + O\!\left(\#C \, x^{1/2} \left(\frac{\log|\Delta_F|}{\#G} + \log x\right)\right) ,
    \end{align*}
    for all $x > 1$.
\end{theorem}
\begin{proof}
    See, e.g.,~\cite[Chapter~2, Section~7]{MR3025442}.
\end{proof}

We also need some results on the degree and the discriminant of certain number fields.

\begin{lemma}\label{lem:kummer-bounds}
    Let $F$ be a number field, let $a \in F^*$ with $a$ not a root of unity, and let $n$ be a positive integer.
    Then:
    \begin{enumerate}
        \item\label{lem:kummer-bounds:ite1} $[F(\zeta_n, a^{1/n}) : \mathbb{Q}] \gg_{F, a} \varphi(n) n$;
        \item\label{lem:kummer-bounds:ite2} $\log|\Delta_{F(\zeta_n, a^{1/n})}| \ll_{F, a} \varphi(n) n \log(2n)$;
        \item\label{lem:kummer-bounds:ite3} Each prime factor of $\Delta_{F(\zeta_n, a^{1/n})}$ divides $\Norm_F(a) \Delta_F n$.
    \end{enumerate}
\end{lemma}
\begin{proof}
    See, e.g., \cite[Lemma 3 and Lemma 5]{MR2314267}.
    (Claim~\ref{lem:kummer-bounds:ite3} is implicit in the proof of \cite[Lemma 5]{MR2314267}.)
\end{proof}

Finally, we need an upper bound for a series involving the Euler totient function.

\begin{lemma}\label{lem:1overphinn-series}
    We have that
    \begin{equation*}
        \sum_{n \,>\, x} \frac1{\varphi(n) n } \ll \frac1{x} ,
    \end{equation*}
    for all $x > 0$.
\end{lemma}
\begin{proof}
    See, e.g., \cite[Theorem~5]{MR4231532}.
\end{proof}

\subsection{Proof of Theorem~\ref{thm:asymptotic}}\label{subsec:proof-thm-asymptotic}

Throughout this section, let $\bm{u} = (u_n)_{n \geq 0}$ be the Lucas sequence defined recursively by $u_0 = 0$, $u_1 = 1$, and $u_n = a_1 u_{n - 1} + a_2 u_{n - 2}$, for every integer $n \geq 2$, where $a_1$ and $a_2$ are fixed nonzero integers.
Let $f_{\bm{u}} := X^2 - a_1 X - a_2$ be the characteristic polynomial of $\bm{u}$ and let $D_{\bm{u}} := a_1^2 + 4 a_2$ be the discriminant of $f_{\bm{u}}$.
Assume that $D_{\bm{u}}$ is not a square in $\mathbb{Z}$, so that $K := \mathbb{Q}(\!\sqrt{D_{\bm{u}}})$ is a quadratic number field.
Let $\alpha, \beta \in K$ be the two roots of $f_{\bm{u}}$, and put $\gamma := \alpha / \beta$.
Note that $\Norm_K(\gamma) = 1$.
Finally, assume that $\bm{u}$ is \emph{nondegenerate}, that is, $\gamma$ is not a root of unity.

For every positive integer $n$, let $K_n := K(\zeta_n, \gamma^{1/n})$.
Note that $K_n / \mathbb{Q}$ is a Galois extension.
Indeed, writing $\gamma = a + b \sqrt{\Delta_K}$ with $a,b \in \mathbb{Q}$, we have that $K_n$ is the splitting field of $X^{2n} - 2a X^n + 1$.
Furthermore, let $C_n \subseteq \Gal(K_n / \mathbb{Q})$ be defined as $C_n := \{\id, \sigma\}$, if there exists $\sigma \in \Gal(K_n / \mathbb{Q})$ such that $\sigma(\zeta_n) = \zeta_n^{-1}$ and $\sigma(\gamma^{1/n}) = \gamma^{-1/n}$, and as $C_n := \{\id\}$, if such $\sigma$ does not exist.
Note that $C_n$ is a union of conjugacy classes, since $\sigma$ belongs to the center of $\Gal(K_n / \mathbb{Q})$.

\begin{lemma}\label{lem:multiplicative-order}
	Let $p$ be a prime number with $p \nmid a_2 \Delta_K$ and let $\mathfrak{p}$ be a prime ideal of $\mathcal{O}_K$ lying over $p$.
	Then $\rho_{\bm{u}}(p)$ is equal to the multiplicative order of $\gamma$ modulo $\mathfrak{p}$.
\end{lemma}
\begin{proof}
	See~\cite[Lemma~5.1]{MR4468150}.
	We remark that \cite[Lemma~5.1]{MR4468150} is stated, incorrectly, with $D_{\bm{u}}$ in place of $\Delta_K$, which makes a difference only for $p = 2$.
\end{proof}

The next lemma is the key tool to the proof of Theorem~\ref{thm:asymptotic}.

\begin{lemma}\label{lem:n-divides-iotaup}
	Let $n$ be a positive integer and let $p$ be a prime number such that $p \nmid a_2 \Delta_K$.
	Then $n \mid \iota_{\bm{u}}(p)$ if and only if $p$ does not ramify in $K_n$ and $\left[K_n \mid p\right] \subseteq C_n$.
\end{lemma}
\begin{proof}
	Suppose that $n \mid \iota_{\bm{u}}(p)$.
	Hence, we have that $p \equiv s \pmod n$, where $s := \big(\Delta_K \mid p\big)$.
	Note that $s \in \{-1, +1\}$, since $p \nmid \Delta_K$.
	In particular, it follows that $p \nmid n$ and so, by Lemma~\ref{lem:kummer-bounds}\ref{lem:kummer-bounds:ite3}, we have that $p$ does not ramify in $K_n$.
	Let $\mathfrak{P}$ be a prime ideal of $K_n$ lying over $p$, and put $\sigma := \left[K_n \mid \mathfrak{B}\right]$.
	Then
	\begin{equation*}
		\sigma(\zeta_n) \equiv \zeta_n^p \equiv \zeta_n^s \pmod {\mathfrak{P}}
	\end{equation*}
	and
	\begin{equation*}
		\sigma(\gamma^{1/n}) \equiv (\gamma^{1/n})^p \equiv \gamma^{(p - s) / n} \cdot \gamma^{s/n} \equiv \gamma^{s/n} \pmod {\mathfrak{P}} ,
	\end{equation*}
	where we used Lemma~\ref{lem:multiplicative-order} and the fact that $\rho_{\bm{u}}(p) \mid (p - s)/n$.
	Moreover, we have that
	\begin{equation*}
		\sigma(\gamma) = \sigma|_K(\gamma) = \left[\frac{K / \mathbb{Q}}{\mathfrak{P} \cap \mathcal{O}_K}\right](\gamma) = \gamma^s ,
	\end{equation*}
	since $\Norm_K(\gamma) = 1$ (and so $\gamma^{-1}$ is the algebraic conjugate of $\gamma$).
	Consequently, we get that $\sigma(\gamma^{1/n}) = \eta \gamma^{s/n}$, for some $n$th root of unity $\eta$.
	Since $p$ does not divide $n$, the polynomial $X^n - 1$ has no multiple roots modulo $\mathfrak{P}$.
	Hence, reduction modulo $\mathfrak{P}$ is injective on the set of $n$th roots of unity.
	Therefore, we get that $\sigma(\zeta_n) = \zeta_n^s$ and $\sigma(\gamma^{1/n}) = \gamma^{s/n}$, which in turn means that $\mathfrak{B} \in C_n$.
	Thus $\left[K_n \mid p\right] \subseteq C_n$, as desired.

	Suppose that $p$ does not ramify in $K_n$ and that $\left[K_n \mid p\right] \subseteq C_n$.
	Let $\mathfrak{P}$ be a prime ideal of $K_n$ lying over $p$, and put $\sigma := \left[K_n \mid \mathfrak{B}\right]$.
	Thus $\sigma(\zeta_n) = \zeta_n^t$ and $\sigma(\gamma^{1/n}) = \gamma^{t/n}$ for some $t \in \{-1, +1\}$.
	Then
	\begin{equation*}
		\zeta_n^t = \sigma(\zeta_n) = \sigma|_{\mathbb{Q}(\zeta_n)}(\zeta_n) = \left[\frac{\mathbb{Q}(\zeta_n) / \mathbb{Q}}{\mathfrak{P} \cap \mathcal{O}_{\mathbb{Q}(\zeta_n)}}\right](\zeta_n) = \zeta_n^p ,
	\end{equation*}
	which implies that $p \equiv t \pmod n$.
	Furthermore, we have that
	\begin{equation*}
		\gamma^{(p - t)/n} \equiv (\gamma^{1/n})^p \cdot \gamma^{-t/n} \equiv \sigma(\gamma^{1/n}) \cdot \gamma^{-t/n} \equiv \gamma^{t/n} \cdot \gamma^{-t/n} \equiv 1 \pmod {\mathfrak{P}} ,
	\end{equation*}
	which, by Lemma~\ref{lem:multiplicative-order}, implies that $\rho_{\bm{u}}(p) \mid (p - t) / n$.
	Hence, we have that $t = \big(\Delta_K \mid p\big)$ and $n \mid \iota_{\bm{u}}(p)$, as desired.
\end{proof}

The rest of the proof follows closely Hooley's proof of Artin's conjecture under the GRH~\cite{MR207630}, with some minor adaptations.
For each positive integer $t$, let us define
\begin{equation}\label{equ:def-deltaut}
    \delta_{\bm{u}}(t) := \sum_{n \,=\, 1}^\infty \frac{\mu(n)\,\#C_{nt}}{[K_{nt} : \mathbb{Q}]} .
\end{equation}
Note that the series in~\eqref{equ:def-deltaut} converges absolutely, thanks to Lemma~\ref{lem:kummer-bounds}\ref{lem:kummer-bounds:ite1} and Lemma~\ref{lem:1overphinn-series}.

For the rest of this section, we shall tacitly ignore the finitely many prime numbers that divide $a_2 \Delta_K$.
For all $x, y, z > 0$, define the sets
\begin{equation*}
    \mathcal{P}_{\bm{u}}(t, x, y) := \big\{p \leq x : t \mid \iota_{\bm{u}}(p) \text{ and } qt \nmid \iota_{\bm{u}}(p) \text{ for every } q \leq y \big\}
\end{equation*}
and
\begin{equation*}
    \mathcal{Q}_{\bm{u}}(t, x, y, z) := \big\{p \leq x : qt \mid \iota_{\bm{u}}(p) \text{ for some } q \in [y, z] \big\} .
\end{equation*}
Moreover, for every $x > 0$, put $y_1 := (\log x) / 6$, $y_2 := x^{1/2} / (\log x)^2$, and $y_3 := x^{1/2} \log x$.
Then, it follows easily that
\begin{equation}\label{equ:4terms}
    \#\mathcal{P}_{\bm{u}}(t, x) = \#\mathcal{P}_{\bm{u}}(t, x, y_1) + O\big(\#\mathcal{Q}_{\bm{u}}(t, x, y_1, y_2) + \#\mathcal{Q}_{\bm{u}}(t, x, y_2, y_3) + \#\mathcal{Q}_{\bm{u}}(t, x, y_3, x) \big) ,
\end{equation}
for all sufficiently large $x$.
The rest of the proof consists of four lemmas estimating the terms of~\eqref{equ:4terms}.

\begin{lemma}\label{lem:main-term}
    Assume the GRH.
    Then
    \begin{equation*}
        \#\mathcal{P}_{\bm{u}}(t, x, y_1) = \delta_{\bm{u}}(t) \, \frac{x}{\log x} + O_{\bm{u}}\!\left(\frac{x}{(\log x)^2}\right) ,
    \end{equation*}
    for all positive integers $t$ and for all $x > t^3$.
\end{lemma}
\begin{proof}
    Let $\mathcal{S}(y_1)$ be the set of all positive squarefree integers whose prime factors are not exceeding $y_1$.
    By the inclusion-exclusion principle and by Lemma~\ref{lem:n-divides-iotaup}, we get that
    \begin{equation}\label{equ:main-term1}
        \#\mathcal{P}_{\bm{u}}(t, x, y_1) = \sum_{n \,\in\, \mathcal{S}(y_1)} \mu(n) \, \#\big\{p \leq x : nt \mid \iota_u(p)\big\} = \sum_{n \,\in\, \mathcal{S}(y_1)} \mu(n)\,\pi_{K_{nt}, C_{nt}}(x) .
    \end{equation}
    Moreover, by Theorem~\ref{thm:chebotarev} and Lemma~\ref{lem:kummer-bounds}\ref{lem:kummer-bounds:ite1} and \ref{lem:kummer-bounds:ite2}, we have that
    \begin{equation}\label{equ:main-term2}
        \sum_{n \,\in\, \mathcal{S}(y_1)} \mu(n)\,\pi_{K_{nt}, C_{nt}}(x)
        = \sum_{n \,\in\, \mathcal{S}(y_1)} \frac{\mu(n)\,\#C_{nt}}{[K_{nt} : \mathbb{Q}]} \frac{x}{\log x} + O\!\left(x^{1/2} \sum_{n \,\in\, \mathcal{S}(y_1)} \log(2nt x) \right) .
    \end{equation}
    If $n \in \mathcal{S}(y_1)$ then $n \leq \prod_{p \,\leq\, y_1} p \leq 4^{y_1} \leq x^{1/3}$ (see~\cite[Lemma~2.8]{MR2919246}).
    Consequently, a fortiori, $\#\mathcal{S}(y_1) \leq x^{1/3}$.
    Therefore, also recalling that $t < x^{1/3}$, we get that
    \begin{equation}\label{equ:main-term3}
        \sum_{n \,\in\, \mathcal{S}(y_1)} \log(2nt x) \ll x^{1/3} \log x .
    \end{equation}
    Furthermore, by~\eqref{equ:def-deltaut}, Lemma~\ref{lem:kummer-bounds}\ref{lem:kummer-bounds:ite1}, and Lemma~\ref{lem:1overphinn-series}, we have that
    \begin{equation}\label{equ:main-term4}
        \delta_{\bm{u}}(t) - \sum_{n \,\in\, \mathcal{S}(y_1)} \frac{\mu(n)\,\#C_{nt}}{[K_{nt} : \mathbb{Q}]} \ll \sum_{n \,>\, y_1} \frac1{[K_{nt} : \mathbb{Q}]} \ll \sum_{n \,>\, y_1} \frac1{\varphi(n)n} \ll \frac1{y_1} \ll \frac1{\log x}.
    \end{equation}
    Putting together \eqref{equ:main-term1}, \eqref{equ:main-term2}, \eqref{equ:main-term3}, and \eqref{equ:main-term4}, the claim follows.
\end{proof}

\begin{lemma}\label{lem:Quxy1y2}
    Assume the GRH.
    Then
    \begin{equation*}
        \#\mathcal{Q}_{\bm{u}}(t, x, y_1, y_2) \ll \frac{x}{(\log x)^2} ,
    \end{equation*}
    for all positive integers $t$ and for all $x > t^3$.
\end{lemma}
\begin{proof}
    If $p \in \mathcal{Q}_{\bm{u}}(t, x, y_1, y_2)$ then $qt \mid \iota_{\bm{u}}(p)$ for some $q \in [y_1, y_2]$.
    Consequently, by Lemma~\ref{lem:n-divides-iotaup}, we have that
    \begin{equation*}
        \#\mathcal{Q}_{\bm{u}}(t, x, y_1, y_2) \leq \sum_{q \,\in\, [y_1, y_2]} \pi_{K_{qt}, C_{qt}}(x) .
    \end{equation*}
    Furthermore, by Theorem~\ref{thm:chebotarev} and Lemma~\ref{lem:kummer-bounds}\ref{lem:kummer-bounds:ite1} and \ref{lem:kummer-bounds:ite2}, we get that
    \begin{align*}
        \sum_{q \,\in\, [y_1, y_2]} \pi_{K_{qt}, C_{qt}}(x)
        &= \sum_{q \,\in\, [y_1, y_2]} \frac{\# C_{qt}}{[K_{qt} : \mathbb{Q}]}\frac{x}{\log x} + O\!\left(x^{1/2} \sum_{q \,\in\, [y_1, y_2]} \left(\frac{\log|\Delta_{K_{qt}}|}{[K_{qt} : \mathbb{Q}]} + \log x\right) \right) \\
        &\ll \sum_{q \,\in\, [y_1, y_2]} \frac1{\varphi(qt)qt}\frac{x}{\log x} + O\!\left(x^{1/2} \sum_{q \,\in\, [y_1, y_2]} \log(2qt x) \right) \\
        &\ll \sum_{q \,\geq\, y_1} \frac1{q^2}\frac{x}{\log x} + O\!\left(x^{1/2} \log x \sum_{q \,\leq \, y_2} 1 \right) \\
        &\ll \frac1{y_1}\frac{x}{\log x} + O\!\left(x^{1/2} \log x \, \frac{y_2}{\log y_2}\right) \\
        &\ll \frac{x}{(\log x)^2} ,
    \end{align*}
    where we used the upper bound $\sum_{q \,\geq\, z} 1 / q^2 \ll 1/z$ and Chebyshev's estimate $\sum_{q \,\leq\, z} 1 \ll z / \log z$, which holds for every $z > 1$.
\end{proof}

\begin{lemma}\label{lem:Quxy2y3}
    We have that
    \begin{equation*}
        \#\mathcal{Q}_{\bm{u}}(t, x, y_2, y_3) \ll \frac{x \log (2\log x)}{\varphi(t)(\log x)^2} ,
    \end{equation*}
    for all positive integers $t$ and for all $x > t^3$.
\end{lemma}
\begin{proof}
    If $p \in \mathcal{Q}_{\bm{u}}(t, x, y_2, y_3)$ then $qt \mid \iota_{\bm{u}}(p)$ for some prime number $q \in [y_2, y_3]$.
    In particular, we have that $p \equiv \pm 1 \pmod {qt}$.
    Consequently, assuming that $x$ is sufficiently large so that $qt \leq y_3 x^{1/3} < x$, by the Brun--Titchmarsh inequality (see, e.g.,~\cite[Theorem~12.7]{MR2919246}), we get that
    \begin{align*}
        \#\mathcal{Q}_{\bm{u}}(t, x, y_2, y_3) &\leq \sum_{q \,\in\, [y_2, y_3]} \#\big\{p \leq x : p \equiv \pm 1 \!\!\!\pmod {qt}\big\}
        \ll \sum_{q \,\in\, [y_2, y_3]} \frac{x}{\varphi(qt) \log\!\big(x / (qt)\big)} \\
        &\ll \frac{x}{\varphi(t)\log x} \sum_{q \,\in\, [y_2, y_3]} \frac1{q} \ll \frac{x\log(2\log x)}{\varphi(t)(\log x)^2} ,
    \end{align*}
    where the last estimate follows from the Mertens theorem.
\end{proof}

\begin{lemma}\label{lem:Quxy3x}
    We have that
    \begin{equation*}
        \#\mathcal{Q}_{\bm{u}}(t, x, y_3, x) \ll_{\bm{u}} \frac{x}{(\log x)^2} ,
    \end{equation*}
    for all positive integers $t$ and for all $x > 1$.
\end{lemma}
\begin{proof}
    If $p \in \mathcal{Q}_{\bm{u}}(t, x, y_3, x)$ then $p \leq x$ and $qt \mid \iota_u(p)$ for some prime number $q \geq y_3$.
    Hence, we have that $\rho_{\bm{u}}(p)$ divides
    \begin{equation*}
        m := \frac{p - \big(D_{\bm{u}} \!\mid p\big)}{q} \leq \frac{2x}{y_3} = \frac{2x^{1/2}}{\log x} ,
    \end{equation*}
    and consequently $p \mid u_m$ (since, in general, $p \mid u_n$ if and only if $p \nmid a_2$ and $\rho_{\bm{u}}(p) \mid n$, see, e.g., \cite[Chapter~1, Section~3]{MR1761897}).
    Therefore, we get that
    \begin{equation*}
        2^{\#\mathcal{Q}_{\bm{u}}(t, x, y_3, x)} \leq \prod_{p \,\in\, \mathcal{Q}_{\bm{u}}(t, x, y_3, x)} p \leq \prod_{m \,\leq\, 2x^{1/2} \!/\!\log x} |u_m| \leq A^{2\sum_{m \,\leq\, 2x^{1/2} \!/\! \log x} m} = A^{O(x / (\log x)^2)} ,
    \end{equation*}
    where $A := \max\{|\alpha|, |\beta|, 2\}$ and where we used the upper bound $|u|_m \leq A^{2m}$, which follows easily from the Binet formula.
    The claim follows.
\end{proof}

At this point, Theorem~\ref{thm:asymptotic} follows by putting together \eqref{equ:4terms} and Lemmas~\ref{lem:main-term}, \ref{lem:Quxy1y2}, \ref{lem:Quxy2y3}, and~\ref{lem:Quxy3x}.
The proof is complete.

% % % % % % % % % % % % % % % % % % % % % % % % % % % % % % % % % % % % % % % % %

\section{Proof of Theorem~\ref{thm:constant}}

\subsection{General preliminaries}

This section collects some general results needed later.

\begin{lemma}\label{lem:Qzetan}
    Let $n > 0$ and $m$ be integers.
    Then we have that:
    \begin{enumerate}

        \item\label{lem:Qzetan:ite1} $\sqrt{m} \in \mathbb{Q}(\zeta_n)$ if and only if $\Delta_{\mathbb{Q}(\!\sqrt{m})} \mid n$;

        \item\label{lem:Qzetan:ite2} if $\sqrt{m} \notin \mathbb{Q}(\zeta_n)$ and $\sqrt{m} \in \mathbb{Q}(\zeta_{2n})$, then $\nu_2(n) \in \{1, 2\}$;

        \item\label{lem:Qzetan:ite3} if $\nu_2(n) = 1$ then $-\zeta_n \in \mathbb{Q}(\zeta_n)^2$;

        \item\label{lem:Qzetan:ite4} if $\nu_2(n) = 2$ then $2\zeta_n \in \mathbb{Q}(\zeta_n)^2$.

    \end{enumerate}
\end{lemma}
\begin{proof}
    Fact \ref{lem:Qzetan:ite1} is well known (cf.\ \cite[Lemma~3]{MR284417}).
    Let us prove \ref{lem:Qzetan:ite2}.
    If $\sqrt{m} \notin \mathbb{Q}(\zeta_n)$ and $\sqrt{m} \in \mathbb{Q}(\zeta_{2n})$ then, by \ref{lem:Qzetan:ite1}, we get that $D \nmid n$ and $D \mid 2n$, where $D := \Delta_{\mathbb{Q}(\!\sqrt{m})}$.
    Let $d$ be the squarefree integer such that $\mathbb{Q}(\!\sqrt{m}) = \mathbb{Q}(\!\sqrt{d})$.
    If $d \equiv 1 \pmod 4$ then $D = d$.
    Hence, $d \nmid n$ and $d \mid 2n$, which is impossible, since $d$ is odd.
    If $d \equiv 2,3 \pmod 4$ then $D = 4d$.
    Hence, $4d \nmid n$ and $2d \mid n$, which implies that $\nu_2(n) = \nu_2(d) + 1 \in \{1, 2\}$, as claimed.
    Finally, \ref{lem:Qzetan:ite3} and \ref{lem:Qzetan:ite4} follow by a quick verification of the identities $-\zeta_n = \left(\zeta_n^{(n + 2) / 4}\right)^2$ and $2 \zeta_n = \left((1 - \mathbf{i}) \zeta_n^{(n + 4) / 8}\right)^2$, respectively.
\end{proof}

\begin{lemma}\label{lem:galois-biquadratic}
    Let $F$ be a field of characteristic zero, and let $X^4 + aX^2 + b \in F[X]$ be an irreducible polynomial with Galois group $G$.
    Then:
    \begin{enumerate}
        \item if $b \in F^2$, then $G \cong C_2 \times C_2$;
        \item if $b \notin F^2$ and $b(a^2 - 4b) \in F^2$, then $G \cong C_4$;
        \item if $b \notin F^2$ and $b(a^2 - 4b) \notin F^2$, then $G \cong D_8$;
    \end{enumerate}
    where $C_n$ and $D_n$ denote the cyclic group and the dihedral groups of $n$ elements, respectively.
\end{lemma}
\begin{proof}
    See \cite[Chapter~V, Section~4, Exercise~9]{MR600654}.
\end{proof}

\begin{lemma}\label{lem:abelian}
    Let $F$ be a field, let $n$ be a positive integer not divisible by the characteristic of $F$, let $m$ be the number of $n$th roots of unity contained in $F$, and let $a \in F$.
    Then $F(\zeta_n, a^{1/n})/F$ is abelian if and only if $a^m \in F^n$.
\end{lemma}
\begin{proof}
    See~\cite[Chapter~8, Theorem~3.2]{MR982265}.
\end{proof}

\begin{lemma}\label{lem:kummer-theory}
	Let $F$ be a number field, let $a \in F^*$, and let $n$ be a positive integer.
	Then $[F(\zeta_n, a^{1/n}) : F(\zeta_n)]$ is equal to the minimum positive integer $\ell$ such that $a^\ell \in F(\zeta_n)^n$.
	Moreover, we have that $f := X^\ell - (a^{1/n})^\ell$ is an irreducible polynomial over $F(\zeta_n)[X]$ and it holds $F(\zeta_n, a^{1/n}) \cong F(\zeta_n)[X] / (f)$.
\end{lemma}
\begin{proof}
	This facts follows from Kummer theory~\cite{MR0219507}.
\end{proof}

\subsection{Preliminaries on a Kummer extension}\label{subsec:a-kummer-extension}

Throughout this section, let $K$ be a quadratic extension of $\mathbb{Q}$ with $\Delta_K \notin \{-3, -4\}$.
Note that the condition on $\Delta_K$ implies that $K$ contains only two roots of unity (namely, $-1$ and $+1$).
This section is devoted to the study of the extension $K(\zeta_n, \gamma^{1/n}) / \mathbb{Q}$, where $n$ is a positive integer and $\gamma \in K$ with $|\!\Norm_K(\gamma)| = 1$.

\begin{lemma}\label{lem:sqrt-in-Kzetan}
    Let $\gamma \in K \setminus (\mathbb{Q} \cup K^2)$ with $\Norm_K(\gamma)\Delta_K \notin \mathbb{Q}^2$, and let $n$ be a positive integer.
    Write $\gamma = a + b\sqrt{\Delta_K}$, with $a,b \in \mathbb{Q}$.
    Then $\sqrt{\gamma} \in K(\zeta_n)$ if and only if $N \in \mathbb{Q}^2$, and $\sqrt{c} \in \mathbb{Q}(\zeta_n)$ or $\sqrt{d} \in \mathbb{Q}(\zeta_n)$, where $N := \Norm_K(\gamma)$, $c := (a - \sqrt{N}) / 2$, and $d := c / \!\Delta_K$.
\end{lemma}
\begin{proof}
    First, note that $f = X^4 - 2aX^2 + N$ is the minimal polynomial of $\sqrt{\gamma}$ over $\mathbb{Q}$.
    Indeed, on the one hand, an easy computation shows that $f(\sqrt{\gamma}) = 0$; while, on the other hand, $\gamma \notin K^2$ implies that $[\mathbb{Q}(\sqrt{\gamma}) : \mathbb{Q}] = 4$, and so the claim follows.
    Let $L$ be the splitting field of $f$ over $\mathbb{Q}$ and put $G := \Gal(L / \mathbb{Q})$.
    Note that $K(\zeta_n)$ is an abelian extension over $\mathbb{Q}$, since it is the compositum of $\mathbb{Q}(\sqrt{\Delta_K})$ and $\mathbb{Q}(\zeta_n)$, which are abelian over $\mathbb{Q}$.

    Suppose that $N \notin \mathbb{Q}^2$.
    Note that $N((-2a)^2 - 4N) = N \Delta_K (2b)^2 \notin \mathbb{Q}^2$, by hypothesis.
    Then, by Lemma~\ref{lem:galois-biquadratic}, we have that $G \cong D_8$ and, in particular, $L / \mathbb{Q}$ is a nonabelian extension.
    Therefore, we have that $\sqrt{\gamma} \notin K(\zeta_n)$.
    Indeed, if $\sqrt{\gamma} \in K(\zeta_n)$ then, since $K(\zeta_n) / \mathbb{Q}$ is Galois, we get that $L \subseteq K(\zeta_n)$, which in turn implies that $L/\mathbb{Q}$ is an abelian extension, but this is absurd.

    Suppose that $N \in \mathbb{Q}^2$.
    Then, by Lemma~\ref{lem:galois-biquadratic}, we have that $G \cong C_2 \times C_2$ and, in particular, $[L : \mathbb{Q}] = |G| = 4$.
    Consequently, we have that $L = \mathbb{Q}(\sqrt{\gamma})$.
    Put $e := (a + \sqrt{N}) / 2$.
    It can be easily checked that $\sqrt{\gamma} = s \sqrt{c} + t\sqrt{e}$ for the right choice of signs $s, t \in \{-1, +1\}$.
    Hence, we have that $L \subseteq \mathbb{Q}(\!\sqrt{c}, \sqrt{e}) = \mathbb{Q}(\!\sqrt{c}, \sqrt{d})$, where the last equality follows from the identity $e = d^{-1} (b/2)^2$ (note that $d \neq 0$ since $\gamma \notin \mathbb{Q}$).
    Furthermore, since $[\mathbb{Q}(\!\sqrt{c}, \sqrt{d}) : \mathbb{Q}] \leq 4$, we get that $L = \mathbb{Q}(\!\sqrt{c}, \sqrt{d})$.

    Suppose that $\sqrt{\gamma} \in K(\zeta_n)$.
    Then, since $K(\zeta_n) / \mathbb{Q}$ is Galois, we have that $L \subseteq K(\zeta_n)$ and so $\sqrt{c} \in K(\zeta_n) = \mathbb{Q}(\zeta_n)(\!\sqrt{\Delta_K})$.
    Hence, $\sqrt{c} = x + y\sqrt{\Delta_K}$ for some $x,y \in \mathbb{Q}(\zeta_n)$.
    If $y = 0$ then $\sqrt{c} = x \in \mathbb{Q}(\zeta_n)$.
    If $x = 0$ and $y \neq 0$, then $\sqrt{d} = \pm \sqrt{c} / \!\sqrt{\Delta_K} = \pm y \in \mathbb{Q}(\zeta_n)$.
    If $x \neq 0$ and $y \neq 0$, then $\sqrt{\Delta_K} = (2xy)^{-1}(c - x^2 - y^2\Delta_K) \in \mathbb{Q}(\zeta_n)$, and so $\sqrt{c} = x + y \sqrt{\Delta_K} \in \mathbb{Q}(\zeta_n)$.
    Hence, in every case, we have that $\sqrt{c} \in \mathbb{Q}(\zeta_n)$ or $\sqrt{d} \in \mathbb{Q}(\zeta_n)$.

    Suppose that $\sqrt{c} \in \mathbb{Q}(\zeta_n)$ or $\sqrt{d} \in \mathbb{Q}(\zeta_n)$.
    If $\sqrt{c} \in \mathbb{Q}(\zeta_n)$ then, recalling that $d = c / \!\Delta_K$ and $K = \mathbb{Q}(\!\sqrt{\Delta_K})$, we get that $\sqrt{c}, \sqrt{d} \in K(\zeta_n)$.
    Therefore, $\sqrt{\gamma} \in L = \mathbb{Q}(\!\sqrt{c}, \sqrt{d}) \subseteq K(\zeta_n)$, so that $\sqrt{\gamma} \in K(\zeta_n)$.
    If $\sqrt{d} \in \mathbb{Q}(\zeta_n)$ then a similar reasoning yields again that $\sqrt{\gamma} \in K(\zeta_n)$.
\end{proof}

\begin{lemma}\label{lem:Qsqrt-in-Kzetan}
    Let $a \in \mathbb{Q} \setminus \mathbb{Q}^2$.
    Then $\sqrt{a} \in K(\zeta_n)$ if and only if $\sqrt{a} \in \mathbb{Q}(\zeta_n)$ or $\sqrt{a / \!\Delta_K} \in \mathbb{Q}(\zeta_n)$.
\end{lemma}
\begin{proof}
    Suppose that $\sqrt{a} \in K(\zeta_n)$.
    Then $\sqrt{a} = x + y \sqrt{\Delta_K}$ for some $x, y \in \mathbb{Q}(\zeta_n)$.
    If $y = 0$ then $\sqrt{a} = x \in \mathbb{Q}(\zeta_n)$.
    If $x = 0$ and $y \neq 0$, then $\sqrt{a / \!\Delta_K} = \pm y \in \mathbb{Q}(\zeta_n)$.
    If $x \neq 0$ and $y \neq 0$, then $\sqrt{\Delta_K} = (2xy)^{-1}(a - x^2 - y^2\Delta_K) \in \mathbb{Q}(\zeta_n)$, and so $\sqrt{a} = x + y \sqrt{\Delta_K} \in \mathbb{Q}(\zeta_n)$.

    Suppose that $\sqrt{a} \in \mathbb{Q}(\zeta_n)$ or $\sqrt{a / \!\Delta_K} \in \mathbb{Q}(\zeta_n)$.
    Then it follows easily that $\sqrt{a} \in K(\zeta_n)$, since $K(\zeta_n) = \mathbb{Q}(\!\sqrt{\Delta_K}, \zeta_n)$.
\end{proof}

\begin{lemma}\label{lem:notin-Kzetan-but-in-Kzeta2n}
    Let $\gamma \in K$ with $|\!\Norm_K(\gamma)| = 1$, and let $n$ be a positive integer.
    If $\sqrt{\gamma} \notin K(\zeta_n)$ and $\sqrt{\gamma} \in K(\zeta_{2n})$, then $\gamma \notin \mathbb{Q} \cup K^2$, $\Norm_K(\gamma) = 1$, and $\nu_2(n) \in \{1, 2\}$.
\end{lemma}
\begin{proof}
    Suppose that $\sqrt{\gamma} \notin K(\zeta_n)$ and $\sqrt{\gamma} \in K(\zeta_{2n})$.
    If $\gamma \in \mathbb{Q}$ then $|\!\Norm_K(\gamma)| = 1$ implies that $\gamma = \pm 1 \in K(\zeta_n)$, which is impossible.
    Hence, we get that $\gamma \notin \mathbb{Q}$.
    Furthermore, from $\sqrt{\gamma} \notin K(\zeta_n)$ it follows that $\gamma \notin K^2$.
    Moreover, we have that $\Norm_K(\gamma) \Delta_K \notin \mathbb{Q}^2$, since $\Norm_K(\gamma) = \pm 1$ and $\Delta_K \neq -4$ by the hypothesis on $K$.
    Hence, we can apply Lemma~\ref{lem:sqrt-in-Kzetan} and, with the same notation of Lemma~\ref{lem:sqrt-in-Kzetan}, we get that $\Norm_K(\gamma) = 1$, $\sqrt{c} \notin \mathbb{Q}(\zeta_n)$, $\sqrt{d} \notin \mathbb{Q}(\zeta_n)$; and $\sqrt{c} \in \mathbb{Q}(\zeta_{2n})$ or $\sqrt{d} \in \mathbb{Q}(\zeta_{2n})$.
    Then the claim follows from Lemma~\ref{lem:Qzetan}\ref{lem:Qzetan:ite2}.
\end{proof}

The proof of the next lemma is similar to that of \cite[Lemma~4]{MR284417}, which characterizes rational numbers in $\mathbb{Q}(\zeta_n)^n$.

\begin{lemma}\label{lem:K-cap-Kzetan-to-n}
    Let $\gamma \in K$ with $|\!\Norm_K(\gamma)| = 1$, and let $n$ be a positive integer.
    Then $\gamma \in K(\zeta_n)^n$ if and only if:
    \begin{enumerate}

        \item\label{ite:power1} $n$ is odd and $\gamma = \delta^n$ for some $\delta \in K$; or

        \item\label{ite:power2} $n$ is even and $\gamma = \delta^{n/2}$ for some $\delta \in K \cap K(\zeta_n)^2$; or

        \item\label{ite:power3} $\nu_2(n) = 2$ and $\gamma = -(2\delta)^{n /2}$ for some $\delta \in K \cap K(\zeta_n)^2$.

    \end{enumerate}
\end{lemma}
\begin{proof}
    First, suppose that $\gamma \in K(\zeta_n)^n$.
    Let us prove that one of \ref{ite:power1}--\ref{ite:power3} holds.
    We have that $\gamma = \varepsilon^n$ for some $\varepsilon \in K(\zeta_n)$.
    Consequently, we get that $K(\zeta_n, \gamma^{1/n}) = K(\zeta_n, \varepsilon) = K(\zeta_n)$,
    which is an abelian extension of $K$.
    Hence, by Lemma~\ref{lem:abelian}, we have that $\gamma^m \in K^n$, where $m$ is the number of $n$th roots of unity in $K$.
    Therefore, there exists $\delta \in K$ such that $\gamma^m = \delta^n$.
    Note that $m = (n, 2)$, since, by hypothesis, $K$ contains only two roots of unity.
    Furthermore, note that $|\!\Norm_K(\delta)| = 1$, since $|\!\Norm_K(\gamma)| = 1$.
    We have to consider several cases.

    If $n$ is odd, then $m = 1$ and we have \ref{ite:power1}.
    Suppose that $n$ is even, so that $\gamma^2 = \delta^n$.
    Therefore, $\gamma = s \delta^{n / 2}$ for some $s \in \{-1, +1\}$.
    Moreover, we have that $(\!\sqrt{\delta})^n = \delta^{n/2} = s\gamma = s\varepsilon^n$.

    If $s = 1$, then $(\!\sqrt{\delta})^n = \varepsilon^n$ and so $\sqrt{\delta} = \eta \varepsilon$, where $\eta$ is a $n$th root of unity.
    Hence, $\sqrt{\delta} \in K(\zeta_n)$ and we have \ref{ite:power2}.
    Suppose that $s = -1$, so that $(\sqrt{\delta})^n = -\varepsilon^n$.
    Hence, $\sqrt{\delta} = \zeta_{2n} \eta \varepsilon$ where $\eta$ is a $n$th root of unity.
    Consequently, we have that $\zeta_{2n} \sqrt{\delta} \in K(\zeta_n)$ and so $\sqrt{\delta} \in K(\zeta_{2n})$.

    Let us prove that $\nu_2(n) \in \{1, 2\}$.
    If $\zeta_{2n} \notin K(\zeta_n)$, then $\sqrt{\delta} \notin K(\zeta_n)$ (otherwise from $\zeta_{2n} \sqrt{\delta} \in K(\zeta_n)$ we would get that $\zeta_{2n} \in K(\zeta_n)$).
    Hence, we have that $\sqrt{\delta} \notin K(\zeta_n)$ and $\sqrt{\delta} \in K(\zeta_{2n})$.
    Therefore, Lemma~\ref{lem:notin-Kzetan-but-in-Kzeta2n} yields that $\nu_2(n) \in \{1, 2\}$.
    If $\zeta_{2n} \in K(\zeta_n)$ then $K(\zeta_{2n}) = K(\zeta_{n})$.
    Thus $[K(\zeta_{2n}) : \mathbb{Q}] = [K(\zeta_n) : \mathbb{Q}]$, which implies that $\sqrt{\Delta_K} \notin \mathbb{Q}(\zeta_n)$ and $\sqrt{\Delta_K} \in \mathbb{Q}(\zeta_{2n})$.
    Therefore, Lemma~\ref{lem:Qzetan}\ref{lem:Qzetan:ite2} yields that $\nu_2(n) \in \{1, 2\}$.
    The claim is proved.

    Recall that $\zeta_{2n} \sqrt{\delta} \in K(\zeta_n)$ and thus $\zeta_n \delta \in K(\zeta_n)^2$.
    If $\nu_2(n) = 1$ then, by Lemma~\ref{lem:Qzetan}\ref{lem:Qzetan:ite3}, we have that $-\zeta_n \in \mathbb{Q}(\zeta_n)^2$, and it follows that $\delta^\prime := -\delta \in K(\zeta_n)^2$.
    Therefore, also using that $n / 2$ is odd, we get that $\gamma = -\delta^{n/2} = (-\delta)^{n / 2} = (\delta^\prime)^{n / 2}$, and we have \ref{ite:power2}.
    If $\nu_2(n) = 2$ then, by Lemma~\ref{lem:Qzetan}\ref{lem:Qzetan:ite4}, we have that $2\zeta_n \in \mathbb{Q}(\zeta_n)^2$, and it follows that $\delta^\prime := \delta/2 \in K(\zeta_n)^2$.
    Hence, we get that $\gamma = -\delta^{n/2} = -(2\delta^\prime)^{n / 2}$, and we have \ref{ite:power3}.

    It remains only to prove that each of \ref{ite:power1}-\ref{ite:power3} implies that $\gamma \in K(\zeta_n)^n$, that is, $\gamma = \varepsilon^n$ for some $\varepsilon \in K(\zeta_n)$.
    In order to do so, it suffices to pick $\varepsilon$ equal to $\delta$, $\sqrt{\delta}$, and $(1+\mathbf{i})\sqrt{\delta}$, for \ref{ite:power1}, \ref{ite:power2}, and \ref{ite:power3}, respectively.
\end{proof}

\begin{lemma}\label{lem:gamma0-repr}
    Every nonzero $\gamma \in K$ can be written as $\gamma = s\gamma_0^h$, where $s \in \{-1, +1\}$, $h$ is a positive integer, and $\gamma_0 \in K$ is not a power in $K$.
    Moreover, this representation is unique except peharps for the sign of $\gamma_0$.
    Furthermore, we have that $\gamma^\ell = z \delta^m$, for some $\delta \in K$ and some integers $\ell, m > 0$ and $z \in \{-1, +1\}$, if and only if $m \mid \ell h$ and
    \begin{enumerate}
        \item\label{lem:gamma0-repr:ite1} $m$ is even, $s^\ell = z$, and $\delta = \pm \gamma_0^{\ell h / m}$; or
        \item\label{lem:gamma0-repr:ite2} $m$ is odd and $\delta = s^\ell z \gamma_0^{\ell h / m}$.
    \end{enumerate}
\end{lemma}
\begin{proof}
    Let $S$ be a finite set of nonequivalent normalized valuations of $K$ containing all the Archimedean valuations and all the valuations $|\cdot|_v$ such that $|\gamma|_v \neq 1$.
    Hence, by construction, $\gamma$ is an $S$-unit of $K$.
    Let $\varepsilon_1, \dots, \varepsilon_w$ be a fundamental system of $S$-units of $K$, where $w := |S| - 1$.
    By the Dirichlet--Chevalley--Hasse Theorem~\cite[Theorem 3.12]{MR2078267}, every $S$-unit of $K$ can be uniquely written as $s \varepsilon_1^{a_1} \cdots \varepsilon_w^{a_w}$, where $s \in \{-1, +1\}$ and $a_1, \dots, a_w \in \mathbb{Z}$.
    If $\gamma = s \varepsilon_1^{a_1} \cdots \varepsilon_w^{a_w}$, then $\gamma = s \gamma_0^h$, where $h := \gcd(a_1, \dots, a_w)$, $b_i := a_i / h$ for $i=1,\dots,w$, and $\gamma_0 := \varepsilon_1^{b_1} \cdots \varepsilon_w^{b_w}$ is not a power in $K$.
    Then the claim on the uniqueness of this representation follows easily.

    Suppose that $\gamma^\ell = z \delta^m$, for some $\delta \in K$ and some integers $\ell, m > 0$ and $z \in \{-1, +1\}$.
    Write $\delta = t \varepsilon_1^{c_1} \cdots \varepsilon_w^{c_w}$ for some $t \in \{-1,+1\}$ and $c_1, \dots, c_w \in \mathbb{Z}$.
    Then $\gamma^\ell = z \delta^m$ if and only if $s^\ell = z t^m$ and $b_i\ell h   = c_i m$ for $i=1,\dots,w$.
    In particular, we have that $m \mid \gcd(b_1 \ell h, \dots, b_w \ell h)$ and so $m \mid \ell h$.
    If $m$ is even and $s^\ell = z$, then the equality $s^\ell = z t^m$ is satisfies for every $t \in \{-1, +1\}$, and so $\delta = \pm \gamma_0^{\ell h / m}$, which is \ref{lem:gamma0-repr:ite1}.
    If $m$ is even and $s^\ell \neq z$, then the equality $s^\ell = z t^m$ is impossible.
    If $m$ is odd then $s^\ell = z t^m$ implies that $t = s^\ell z$ and so $\delta = s^\ell z \gamma_0^{\ell h / m}$, which is \ref{lem:gamma0-repr:ite2}.

    Vice versa, if $\delta \in K$ and $\ell, m > 0$ and $z \in \{-1, +1\}$ are integers such that $m \mid \ell h$ and either \ref{lem:gamma0-repr:ite1} or \ref{lem:gamma0-repr:ite2} holds, then it follows easily that $\gamma^\ell = z\delta^m$.
\end{proof}

The first part of the proof of the next lemma follows a strategy similar to the proof of \cite[Lemma~5]{MR284417}.

\begin{lemma}\label{lem:degree-and-existence}
    Let $\gamma \in K$ with $|\!\Norm_K(\gamma)| = 1$.
    Write $\gamma = s\gamma_0^h$, where $s \in \{-1, +1\}$, $h$ is a positive integer, and $\gamma_0 \in K$ is not a power in $K$ (see Lemma~\ref{lem:gamma0-repr}).
    Let $n$ be a positive integer and put $n^\prime := n / (n, 2h)$ and $h^\prime := 2h / (n, 2h)$.
    Then we have that
    \begin{equation}\label{equ:degree-Kzetangamma1n}
        \big[K(\zeta_n, \gamma^{1/n}) : K(\zeta_n) \big] = n^\prime \cdot
        \begin{cases}
            1 & \text{ if one of \ref{ite:C1}--\ref{ite:C4} holds; } \\
            2 & \text{ otherwise; }
        \end{cases}
    \end{equation}
    where the conditions in~\eqref{equ:degree-Kzetangamma1n} are the following
    \begin{enumerate}[label=(C\arabic*)]

		\item\label{ite:C1} $n$ is odd;

		\item\label{ite:C2} $s^{n^\prime} = 1$, $n$ is even, and $\gamma_0^{h^\prime} \in K(\zeta_n)^2$;

		\item\label{ite:C3} $s = -1$, $\nu_2(n) = 1$, and $-\gamma_0^{h^\prime} \in K(\zeta_n)^2$;

		\item\label{ite:C4} $s^{n^\prime} = -1$, $\nu_2(n) = 2$, and $2\gamma_0^{h^\prime} \in K(\zeta_n)^2$.

    \end{enumerate}
	Furthermore, let $\sigma_1 \in \Gal(K(\zeta_n) / \mathbb{Q})$ be the complex conjugation and, if $\sqrt{\Delta_K} \notin \mathbb{Q}(\zeta_n)$, let $\sigma_2 \in \Gal(K(\zeta_n) / \mathbb{Q})$ be the unique automorphism satisfying $\sigma_2(\zeta_n) = \zeta_n^{-1}$ and $\sigma_2(\!\sqrt{\Delta_k}) = -\overline{\sqrt{\Delta_k}}$, otherwise let $\sigma_2 := \sigma_1$.
	Then there exists $\sigma \in \Gal(K(\zeta_n, \gamma^{1/n}) / \mathbb{Q})$ such that $\sigma(\zeta_n) = \zeta_n^{-1}$ and $\sigma(\gamma^{1/n}) = \gamma^{-1/n}$ if and only if
	\begin{enumerate}[label=(D\arabic*)]

		\item\label{ite:D1} it holds \ref{ite:C1} and $\sigma_i\big(\gamma_0^{h^\prime\!/2}\big) = \gamma_0^{-h^\prime\!/2}$ for some $i \in \{1,2\}$;

		\item\label{ite:D2} it holds \ref{ite:C2} and $\sigma_i\!\Big(\!\sqrt{\smash[b]{\gamma_0^{h^\prime}}}\Big) = \sqrt{\smash[b]{\gamma_0^{h^\prime}}}^{-1}$ for some $i \in \{1,2\}$;

		\item\label{ite:D3} it holds \ref{ite:C3} and $\sigma_i\!\Big(\!\sqrt{\smash[b]{-\gamma_0^{h^\prime}}}\Big) = \sqrt{\smash[b]{-\gamma_0^{h^\prime}}}^{-1}$ for some $i \in \{1,2\}$;

		\item\label{ite:D4} it holds \ref{ite:C4} and $\sigma_i\!\Big(\!\sqrt{\smash[b]{2\gamma_0^{h^\prime}}}\Big) = 2\sqrt{\smash[b]{2\gamma_0^{h^\prime}}}^{-1}$ for some $i \in \{1,2\}$;

		\item\label{ite:D5} neither of \ref{ite:C1}--\ref{ite:C4} holds and $\sigma_i\big(\gamma_0^{h^\prime}\big) = \gamma_0^{-h^\prime}$ for some $i \in \{1,2\}$.

	\end{enumerate}
\end{lemma}
\begin{proof}
    By Lemma~\ref{lem:kummer-theory}, we have that $\big[K(\zeta_n, \gamma^{1/n}) : K(\zeta_n) \big]$ is equal to the least positive integer $\ell$ such that $\gamma^\ell \in K(\zeta_n)^n$.

    First, let us prove that $\ell \in \{n^\prime, 2n^\prime\}$.
    On the one hand, by Lemma~\ref{lem:K-cap-Kzetan-to-n}, $\gamma^\ell \in K(\zeta_n)^n$ implies that $\gamma^{2\ell} \in K^n$, which in turn, by Lemma~\ref{lem:gamma0-repr}, yields that $n \mid 2h \ell$, and so $n^\prime \mid \ell$.
    On the other hand, we have that $\gamma^{2n^\prime} = (\gamma_0^{h^\prime})^n \in K^n \subseteq K(\zeta_n)^n$, so that $\ell \leq 2n^\prime$.
    The claim is proved.

    At this point, by applying Lemma~\ref{lem:K-cap-Kzetan-to-n} and Lemma~\ref{lem:gamma0-repr} (and also employing the fact that $-1 \in K(\zeta_n)^2$ when $4 \mid n$), one can prove that $\gamma^{n^\prime} \in K(\zeta_n)^n$ is equivalent to ``\ref{ite:C1} or \ref{ite:C2} or \ref{ite:C3} or \ref{ite:C4}'', which gives \eqref{equ:degree-Kzetangamma1n}.

	It remains to prove the statement on the existence of $\sigma$.
	Thanks to Lemma~\ref{lem:kummer-theory}, we have that $K(\zeta_n, \gamma^{1/n}) \cong K(\zeta_n)[X] / (f)$ where $f \in K(\zeta_n)[X]$ is equal to $X^{n^\prime} - \eta s \gamma_0^{h^\prime\!/ 2}$, $X^{n^\prime} - \eta \sqrt{\smash[b]{\gamma_0^{h^\prime}}}$, $X^{n^\prime} - \eta \sqrt{\smash[b]{-\gamma_0^{h^\prime}}}$, $X^{n^\prime} - \eta 2^{-1}(1 + \mathbf{i})\sqrt{\smash[b]{2\gamma_0^{h^\prime}}}$, or $X^{2n^\prime} - \eta \gamma^{h^\prime}$, for some $n$th root of unity $\eta$, if it holds \ref{ite:C1}, \ref{ite:C2}, \ref{ite:C3}, \ref{ite:C4}, or none of them, respectively.
	We consider only the case in which \ref{ite:C3} holds, since the proofs of the other cases are very similar.
	Suppose that \ref{ite:C3} holds.
	We have to prove that $\sigma$ exists if and only if there exists $\sigma_0 \in \Gal(K(\zeta_n) / \mathbb{Q})$ such that $\sigma_0(\zeta_n) = \zeta_n^{-1}$ and $\sigma_0\big(\!\sqrt{\smash[b]{-\gamma_0^{h^\prime}}}\big) \sqrt{\smash[b]{-\gamma_0^{h^\prime}}} = 1$ (note that $\sigma_0 \in \{\sigma_1, \sigma_2\}$).

	Suppose that $\sigma$ exists.
	Note that $\sqrt{\smash[b]{-\gamma_0^{h^\prime}}} = \rho(\gamma^{1/n})^{n^\prime}$ for some $n$th root of unity $\rho$.
	Then, letting $\sigma_0 := \sigma|_{K(\zeta_n)}$, we have that $\sigma_0(\zeta_n) = \zeta_n^{-1}$ and
	\begin{equation*}
		\sigma_0\!\Big(\!\sqrt{\smash[b]{-\gamma_0^{h^\prime}}}\Big) = \sigma\big(\rho(\gamma^{1/n})^{n^\prime}\big) = \sigma(\rho)\, \sigma(\gamma^{1/n})^{n^\prime} = \rho^{-1} (\gamma^{1/n})^{-n^\prime} = \sqrt{\smash[b]{-\gamma_0^{h^\prime}}}^{-1} ,
	\end{equation*}
	as desired.

	Vice versa, suppose that $\sigma_0$ exists.
	Then $\sigma_0$ can be extended to an automorphism $\widetilde{\sigma} \in \Gal(K(\zeta_n, \gamma^{1/n})/\mathbb{Q})$ that sends the root $\gamma^{1/n}$ of $f$ to the root $\gamma^{-1/n}$ of
	\begin{equation*}
		\sigma_0 f = X^{n^\prime} - \sigma_0\!\Big(\!\eta \sqrt{\smash[b]{-\gamma_0^{h^\prime}}}\Big)
		= X^{n^\prime} - \eta^{-1} \sqrt{\smash[b]{-\gamma_0^{h^\prime}}}^{-1} ,
	\end{equation*}
	and we can take $\sigma := \widetilde{\sigma}$.
	The proof is complete.
\end{proof}

\begin{lemma}\label{lem:sigma-sqrt}
	Let $\gamma \in K \setminus (\mathbb{Q} \cup K^2)$ with $\Norm_K(\gamma)\Delta_K \notin \mathbb{Q}^2$, and let $n$ be a positive integer.
	Write $\gamma = a + b\sqrt{\Delta_K}$, with $a,b \in \mathbb{Q}$.
	Suppose that $\sqrt{\gamma} \in K(\zeta_n)$ and let $N := \Norm_K(\gamma)$, $c := (a - \sqrt{N}) / 2$, and $d := c / \!\Delta_K$ (Note that $N \in \mathbb{Q}^2$ by Lemma~\ref{lem:sqrt-in-Kzetan}).
	Suppose that $\sqrt{\Delta_K} \notin \mathbb{Q}(\zeta_n)$ and let $\sigma \in \Gal(K(\zeta_n) / \mathbb{Q})$ be the unique automorphism such that $\sigma(\zeta_n) = \zeta_n^{-1}$ and $\sigma(\!\sqrt{\Delta_K}) = -\overline{\sqrt{\Delta_K}}$.
	Then
	\begin{equation*}
		\sigma(\!\sqrt{\gamma}) \sqrt{\gamma} =
		\begin{cases}
			\sqrt{N} & \text{if $\Delta_K > 0$ and $\big(\big(\sqrt{c} \in \mathbb{Q}(\zeta_n), c < 0\big) \text{ or } \big(\sqrt{d} \in \mathbb{Q}(\zeta_n), d > 0\big)\big)$;} \\
			-\sqrt{N} & \text{if $\Delta_K > 0$ and $\big(\big(\sqrt{c} \in \mathbb{Q}(\zeta_n), c > 0\big) \text{ or } \big(\sqrt{d} \in \mathbb{Q}(\zeta_n), d < 0\big)\big)$;} \\
			\gamma & \text{if $\Delta_K < 0$ and $\big(\big(\sqrt{c} \in \mathbb{Q}(\zeta_n), c > 0\big) \text{ or } \big(\sqrt{d} \in \mathbb{Q}(\zeta_n), d > 0\big)\big)$;} \\
			-\gamma & \text{if $\Delta_K < 0$ and $\big(\big(\sqrt{c} \in \mathbb{Q}(\zeta_n), c < 0\big) \text{ or } \big(\sqrt{d} \in \mathbb{Q}(\zeta_n), d < 0\big)\big)$.}
		\end{cases}
	\end{equation*}
\end{lemma}
\begin{proof}
	Since $\sqrt{\gamma} \in K(\zeta_n)$, by Lemma~\ref{lem:sqrt-in-Kzetan} we get that $\sqrt{c} \in \mathbb{Q}(\zeta_n)$ or $\sqrt{d} \in \mathbb{Q}(\zeta_n)$.
	Suppose that $\sqrt{c} \in \mathbb{Q}(\zeta_n)$.
	It can be easily checked that $\sqrt{\gamma} = s\big(\sqrt{c} + b \sqrt{\Delta_K} / (2\sqrt{c})\big)$, where $s \in \{-1, +1\}$.
	Hence, we get that
	\begin{equation*}
		\sigma(\!\sqrt{\gamma}) \gamma
		= \left(\overline{\sqrt{c}} - \frac{b\,\overline{\sqrt{\Delta_K}}}{2 \overline{\sqrt{c}}} \right) \left(\sqrt{c} - \frac{b\sqrt{\Delta_K}}{2 \sqrt{c}} \right)
		= |c| - \frac{b^2|\Delta_K|}{4|c|} + \frac{b}{2} \left(\frac{\overline{\sqrt{c}}}{\sqrt{c}} \sqrt{\Delta_K} - \frac{\sqrt{c}}{\overline{\sqrt{c}}} \overline{\sqrt{\Delta_K}}\right) .
	\end{equation*}
	The claim follows by considering the possible signs of $c$ and $\Delta_K$.
	For instance, if $c > 0$ and $\Delta_K < 0$, then
	\begin{equation*}
		\sigma(\!\sqrt{\gamma}) \gamma
		= c + \frac{b^2\Delta_K}{4c} + \frac{b}{2} \left(\sqrt{\Delta_K} - \overline{\sqrt{\Delta_K}}\right) = a + b\sqrt{\Delta_K} = \gamma .
	\end{equation*}
	The case $\sqrt{d} \in \mathbb{Q}(\zeta_n)$ is handled similarly, using the identity $\sqrt{\gamma} = s\big(b / (2\sqrt{d}) + \sqrt{d}\sqrt{\Delta_K}\big)$, where $s \in \{-1, +1\}$.
\end{proof}

\subsection{Wagstaff sums}

Let $h$, $m$, $t$ be positive integers.
Define the \emph{Wagstaff sum}~\cite{MR674829}
\begin{equation*}
    S_{h,m}(t) := \sum_{\substack{n \,=\, 1 \\[1pt] m \,\mid\, nt}}^\infty \frac{\mu(n) (nt, h)}{\varphi(nt) nt} .
\end{equation*}
(We adopt a notation different from~\cite{MR674829} since we are mostly interested in considering the Wagstaff sum as a function of $t$, for fixed $h$ and $m$.)
In this section, we provide an explicit formula for $S_{h,m}(t)$ as a product involving the Artin constant, a multiplicative function, and a periodic function, with both these functions having rational values.
Formulas for $S_{h,m}(t)$ as a rational multiple of the Artin constant were already proved by Wagstaff~\cite{MR674829} (see also~\cite[Section~3]{MR3086966}), and our result can be deduced from them.
However, we believe it is easier to give an independent proof.
Let
\begin{equation*}
    B_h := 2 \prod_{\substack{p \,>\, 2 \\ p \,\mid\, h}} \left(1 - \frac{p-1}{p^2 - p - 1}\right) \quad\text{ and }\quad f_h(t, n) := \frac{(nt, h)\varphi(t)}{(t, h)\varphi(nt)n} ,
\end{equation*}
for every positive integer $n$.
Note that $f_h(t, \cdot)$ is a multiplicative function and
\begin{equation*}
    f_h(t, p) =
    \begin{cases}
        p^{-1} & \text{ if $p \mid t$ and $\nu_p(t) < \nu_p(h)$;} \\
        p^{-2} & \text{ if $p \mid t$ and $\nu_p(t) \geq \nu_p(h)$;} \\
        (p - 1)^{-1} & \text{ if $p \nmid t$ and $p \mid h$;} \\
        (p(p - 1))^{-1} & \text{ if $p \nmid t$ and $p \nmid h$;} \\
    \end{cases}
\end{equation*}
for each prime number $p$.
Furthermore, put
\begin{equation*}
    F_h(t) := \frac{(t, h)}{\varphi(t)t} \widetilde{F}_h(t)
    \quad\text{ and }\quad \widetilde{F}_h(t) := \prod_{\substack{p \,>\, 2 \\ p \,\mid\, t}} \left(1 - \frac{(pt, h)}{p^2(t, h)}\right) \left(1 - \frac{(p, h)}{p(p-1)}\right)^{-1} ,
\end{equation*}
while $G_{h,m}(t) := \widetilde{G}_{h, m / (m, t)}(t)$ and
\begin{equation*}
    \widetilde{G}_{h, m}(t) := |\mu(m)|\prod_{\substack{p \,>\, 2 \\ p \,\mid\, m}} \left(1 - \frac1{f_{h}(t, p)}\right)^{-1}
    \cdot
    \begin{cases}
        1 - f_h(t, 2) & \text{ if $2 \nmid m$;} \\
        {-}f_h(t, 2) & \text{ if $2 \mid m$.} \\
    \end{cases}
\end{equation*}
Note that $F_h(\cdot)$ is a multiplicative function, and that $G_{h,m}(\cdot)$ is a periodic function.

\begin{lemma}\label{lem:wagstaff-sum}
    We have that
    \begin{equation*}
        S_{h,m}(t) = A \, B_h\, F_h(t) \, G_{h,m}(t) ,
    \end{equation*}
    for all positive integers $h$, $m$, $t$.
\end{lemma}
\begin{proof}
    For each prime number $p > 2$, put
    \begin{equation*}
        g_h(p) :=
        \begin{cases}
            \big(1 - (p-1)^{-1}\big)^{-1} & \text{ if $p \mid h$;} \\
            \big(1 - (p(p-1))^{-1}\big)^{-1} & \text{ if $p \nmid h$.} \\
        \end{cases}
    \end{equation*}
    Since $f_h(t, \cdot)$ is multiplicative, we have that
    \begin{align*}
        T_h(t) &:= \sum_{\substack{n \,=\, 1 \\ 2 \,\nmid\, n}}^\infty \mu(n) f_h(t, n)
        = \prod_{p \,>\, 2} \big(1 - f_h(t, p)\big) \\
        &= \prod_{\substack{p \,>\, 2 \\[1pt] p \,\nmid\, t,\; p \,\nmid\, h}} \left(1 - \frac1{p(p-1)}\right) \prod_{\substack{p \,>\, 2 \\[1pt] p \,\nmid\, t,\; p \,\mid\, h}} \left(1 - \frac1{p-1}\right) \prod_{\substack{p \,>\, 2 \\ p \,\mid\, t}} \big(1 - f_h(t, p)\big) \\
        &= \prod_{\substack{p \,>\, 2 \\[1pt] p \,\nmid\, h}} \left(1 - \frac1{p(p-1)}\right) \prod_{\substack{p \,>\, 2 \\[1pt] p \,\mid\, h}} \left(1 - \frac1{p-1}\right) \prod_{\substack{p \,>\, 2 \\ p \,\mid\, t}} \big(1 - f_h(t, p)\big) g_h(p) \\
        &= \prod_{\substack{p \,>\, 2}} \left(1 - \frac1{p(p-1)}\right) \prod_{\substack{p \,>\, 2 \\[1pt] p \,\mid\, h}} \left(1 - \frac1{p-1}\right) \left(1 - \frac1{p(p-1)}\right)^{-1}  \widetilde{F}_h(t) \\
        &= A\, B_h \, \widetilde{F}_h(t) .
    \end{align*}
    Consequently, letting
    \begin{equation*}
        \widetilde{S}_{h, m}(t) := \sum_{\substack{n \,=\, 1 \\ m \,\mid\, n}}^\infty \mu(n) f_h(t, n) ,
    \end{equation*}
    we get that
    \begin{equation}\label{equ:wagstaff1}
        \widetilde{S}_{h,1}(t) = T_h(t) \big(1 - f_h(t, 2)\big) = A\, B_h \, \widetilde{F}_h(t) \big(1 - f_h(t, 2)\big)
    \end{equation}
    and
    \begin{equation}\label{equ:wagstaff2}
        \widetilde{S}_{h,2}(t) = \widetilde{S}_{h,1}(t) - T_h(t) = A\, B_h \, \widetilde{F}_h(t) \cdot \big({-f_h(t, 2)}\big) .
    \end{equation}
    If $m$ is not squarefree, then it is clear that $\widetilde{S}_{h,m}(t) = 0$.
    Suppose that $m$ is squarefree and that $p$ is an odd prime factor of $m$.
    Then
    \begin{align*}
        \widetilde{S}_{h,m}(t) &= \sum_{\substack{n^\prime \,=\, 1 \\[1pt] m/p \,\mid\, n^\prime, \; p \,\nmid\, n^\prime}}^\infty \mu(pn^\prime) f_h(t, pn^\prime)
        = -f_h(t, p) \sum_{\substack{n^\prime \,=\, 1 \\[1pt] m/p \,\mid\, n^\prime, \; p \,\nmid\, n^\prime}}^\infty \mu(n^\prime) f_h(t, n^\prime) \\
        &= -f_h(t, p) \sum_{\substack{n^\prime \,=\, 1 \\[1pt] m/p \,\mid\, n^\prime, \; m \,\nmid\, n^\prime}}^\infty \mu(n^\prime) f_h(t, n^\prime)
        = -f_h(t, p) \big(\widetilde{S}_{h,m / p}(t) - \widetilde{S}_{h,m}(t) \big) ,
    \end{align*}
    from which it follows that
    \begin{equation}\label{equ:wagstaff3}
        \widetilde{S}_{h,m}(t) = \widetilde{S}_{h,m / p}(t) \left(1 - \frac1{f_h(t, p)}\right)^{-1}  .
    \end{equation}
    Therefore, from~\eqref{equ:wagstaff1}, \eqref{equ:wagstaff2}, and~\eqref{equ:wagstaff3}, we get that
    \begin{equation*}
        \widetilde{S}_{h,m}(t) = A\, B_h \, \widetilde{F}_h(t) \, \widetilde{G}_{h,m}(t) .
    \end{equation*}
    In conclusion, by noticing that
    \begin{equation*}
        S_{h,m}(t) = \frac{(t, h)}{\varphi(t)t} \widetilde{S}_{h,m / (m, t)}(t) ,
    \end{equation*}
    the claim follows.
\end{proof}

\subsection{Proof of Theorem~\ref{thm:constant}}

We employ the same notation of Section~\ref{subsec:proof-thm-asymptotic}.
Furthermore, we assume that $\Delta_K \notin \{-3, -4\}$, in order to apply to $K$ the results of Section~\ref{subsec:a-kummer-extension}.
By Lemma~\ref{lem:gamma0-repr}, there exists $s \in \{-1, +1\}$, a positive integer $h$, and $\gamma_0 \in K$ which is not a power in $K$, such that $\gamma = s \gamma_0^h$.
For every positive integer $n$, we have that
\begin{equation*}
	[K_n : \mathbb{Q}] = [K_n : K(\zeta_n)] [K(\zeta_n) : \mathbb{Q}] = [K_n : K(\zeta_n)] \,\varphi(n)
		\begin{cases}
			2 & \text{ if } \sqrt{\Delta_K} \notin \mathbb{Q}(\zeta_n) ; \\
			1 & \text{otherwise} .
		\end{cases}
\end{equation*}
Furthermore, in light of Lemma~\ref{lem:Qzetan}\ref{lem:Qzetan:ite1} and Lemma~\ref{lem:sigma-sqrt}, the conditions \ref{ite:C1}--\ref{ite:C4} and \ref{ite:D1}--\ref{ite:D5} of Lemma~\ref{lem:degree-and-existence} depend only on the divisibility of $n$ by some integers determined only by $s$, $h$, and $\gamma_0$.
Consequently, there exist $c_1, \dots, c_k \in \mathbb{Q}$ and $m_1, \dots, m_k \in \mathbb{Z}^+$, depending only on $s$, $h$, and $\gamma_0$, such that
\begin{equation}\label{equ:lin-comb}
	\frac{\# C_n}{[K_n : \mathbb{Q}]} = \frac{(n, 2h)}{\varphi(n) n} \sum_{i \,=\, 1}^k c_i \chi_{m_i}(n) ,
\end{equation}
for every positive integer $n$, where $\chi_m(\cdot)$ is the characteristic function of $m\mathbb{Z}$.
For every positive integer $t$, let
\begin{equation}\label{equ:def-Gu}
	G_{\bm{u}}(t) := B_{2h} \sum_{i \,=\, 1}^k c_i G_{2h,m_i}(t) .
\end{equation}
Note that $G_{\bm{u}}(\cdot)$ is a periodic function.
Then, by~\eqref{equ:def-deltaut}, \eqref{equ:lin-comb}, and Lemma~\ref{lem:wagstaff-sum}, we get that
\begin{align*}
    \delta_{\bm{u}}(t) &:= \sum_{n \,=\, 1}^\infty \frac{\mu(n)\,\#C_{nt}}{[K_{nt} : \mathbb{Q}]}
    = \sum_{n \,=\, 1}^\infty \frac{\mu(n)(nt, 2h)}{\varphi(nt) nt} \sum_{i \,=\, 1}^k c_i \chi_{m_i}(nt)
    = \sum_{i \,=\, 1}^k c_i \sum_{\substack{n \,=\, 1 \\ m_i \,\mid\, nt}}^\infty \frac{\mu(n)(nt, 2h)}{\varphi(nt) nt} \\
    &= \sum_{i \,=\, 1}^k c_i S_{2h, m_i} (t)
    = A B_{2h}\, F_{2h}(t) \sum_{i \,=\, 1}^k c_i G_{2h,m_i}(t) = A \, F_{2h}(t) \, G_{\bm{u}}(t) ,
\end{align*}
for every positive integer $t$.
The proof is complete.

\section{Examples}\label{sec:examples}

We employ the same notation of Section~\ref{subsec:proof-thm-asymptotic}.
We provide only the main details and leave the rest of the computations to the reader.

\subsection{Fibonacci numbers (Example~\ref{exe:11})}

Let $a_1 = a_2 = 1$, so that $\bm{u}$ is the sequence of Fibonacci numbers.
Then $K = \mathbb{Q}(\!\sqrt{5})$, $s = -1$, $h = 2$, and $\gamma_0 = \tfrac1{2} + \tfrac1{2}\sqrt{5}$.
Since $\Norm_K(\gamma_0) = -1$, it follows from Lemma~\ref{lem:sqrt-in-Kzetan} that neither of $\sqrt{\gamma_0}$, $\sqrt{-\gamma_0}$, $\sqrt{2\gamma_0}$ belongs to $K(\zeta_n)$, for every positive integer $n$.
In turn, with the aid of Lemma~\ref{lem:Qsqrt-in-Kzetan} and Lemma~\ref{lem:Qzetan}\ref{lem:Qzetan:ite1}, this implies that \ref{ite:C2}, \ref{ite:C3}, and \ref{ite:C4} cannot occur.
Furthermore, with the notation of Lemma~\ref{lem:degree-and-existence}, we have that $\sigma_1(\gamma_0) \gamma_0 = |\gamma_0|^2$ is not a root of unity and, if $\sqrt{5} \notin \mathbb{Q}(\zeta_n)$, then $\sigma_2(\gamma_0) = -\gamma_0^{-1}$.
Therefore, Lemma~\ref{lem:degree-and-existence} yields that
\begin{equation*}
	[K_n : \mathbb{Q}] = \frac{\varphi(n) n} {(n, 4)} \cdot \begin{cases} 2 & \text{ if } 2 \mid n \\ 1 & \text{ if } 2 \nmid n \end{cases} \cdot \begin{cases} 1 & \text{ if } 5 \mid n \\ 2 & \text{ if } 5 \nmid n \end{cases}
\end{equation*}
and
\begin{equation*}
	\# C_n =
		\begin{cases}
			2 & \text{ if $4 \nmid n$ and $5 \nmid n$}; \\
			1 & \text{ otherwise}.
		\end{cases}
\end{equation*}
Consequently, we have that
\begin{equation*}
	\frac{\# C_n}{[K_n : \mathbb{Q}]} = \frac{(n, 4)}{\varphi(n)n} \left(\chi_1(n) - \tfrac1{2}\chi_2(n) - \tfrac1{4}\chi_4(n) + \tfrac1{4}\chi_{20}(n)\right)
\end{equation*}
At this point, the claim follows from \eqref{equ:lin-comb} and \eqref{equ:def-Gu}.

\subsection{Example~\ref{exe:4-1}}

Let $a_1 = 4$ and $a_2 = -1$.
Then $K = \mathbb{Q}(\!\sqrt{3})$, $s = 1$, $h = 2$, and $\gamma_0 = 2 + \sqrt{3}$.
Since $\Norm_K(\gamma_0) = 1$, by Lemma~\ref{lem:sqrt-in-Kzetan}, we have that $\sqrt{\gamma_0} \in K(\zeta_n)$ if and only if $8 \mid n$ or $24 \mid n$, for every positive integer $n$.
We some patience, one can work out that the conditions of Lemma~\ref{lem:degree-and-existence} are equivalent to the following:
\begin{enumerate}[label=(C\arabic*$^\prime$)]
    \item\label{ite:C1prime} $2 \nmid n$;
    \item\label{ite:C2prime} $\nu_2(n) = 1$ or $8 \mid n$ or $24 \mid n$;
    \item $\bot$;
    \item $\bot$;
    \item[(D1$^\prime$)] $2 \nmid n$;
    \item[(D2$^\prime$)] it holds \ref{ite:C2prime}, and $4 \nmid n$ or $24 \nmid n$;
    \item[(D3$^\prime$)] $\bot$;
    \item[(D4$^\prime$)] $\bot$;
    \item[(D5$^\prime$)] neither \ref{ite:C1prime} nor \ref{ite:C2prime};
\end{enumerate}
where $\bot$ denotes a condition that is never satisfied.
Consequently, one gets that
\begin{equation*}
    \frac{\# C_n}{[K_n : \mathbb{Q}]} = \frac{(n, 4)}{\varphi(n)n} \left(\chi_1(n) - \tfrac1{2}\chi_4(n) + \tfrac1{2}\chi_{24}(n)\right) .
\end{equation*}
Then the claim follows from \eqref{equ:lin-comb} and \eqref{equ:def-Gu}.

\subsection{Example~\ref{exe:102}}

Let $a_1 = 10$ and $a_2 = 2$.
Then $K = \mathbb{Q}(\!\sqrt{3})$, $s = -1$, $h = 3$, and $\gamma_0 = 2 + \sqrt{3}$.
With some effort, one finds that the conditions of Lemma~\ref{lem:degree-and-existence} are equivalent to the following:
\begin{enumerate}[label=(C\arabic*$^{\prime\prime}$)]
    \item\label{ite:C1primeprime} $2 \nmid n$;
    \item\label{ite:C2primeprime} $8 \mid n$ or $24 \mid n$;
    \item $\bot$;
    \item $\bot$;
    \item[(D1$^{\prime\prime}$)] $2 \nmid n$;
    \item[(D2$^{\prime\prime}$)] $\bot$;
    \item[(D3$^{\prime\prime}$)] $\bot$;
    \item[(D4$^{\prime\prime}$)] $\bot$;
    \item[(D5$^{\prime\prime}$)] neither \ref{ite:C1primeprime} nor \ref{ite:C2primeprime}.
\end{enumerate}
Consequently, one gets that
\begin{equation*}
    \frac{\# C_n}{[K_n : \mathbb{Q}]} = \frac{(n, 6)}{\varphi(n)n} \left(\chi_1(n) - \tfrac1{2}\chi_2(n) + \tfrac1{2}\chi_{24}(n)\right) .
\end{equation*}
Then the claim follows from \eqref{equ:lin-comb} and \eqref{equ:def-Gu}.

% % % % % % % % % % % % % % % % % % % % % % % % % % % % % % % % % % % % % % % % %

\section{Tables}
\FloatBarrier

\begin{center}
    \begin{table}[h]
        \begin{tabular}{cccc|cccc}
            \toprule
            $t$ & $\delta_{\bm{u}}(t)$ & $\widetilde{\delta}_{\bm{u}}(t)$ & error & $t$ & $\delta_{\bm{u}}(t)$ & $\widetilde{\delta}_{\bm{u}}(t)$ & error \\
            \midrule
            1 & 0.373956 & 0.374149 & 0.052\% & 21 & 0.001588 & 0.001621 & 2.053\% \\
            2 & 0.285387 & 0.285535 & 0.052\% & 22 & 0.002597 & 0.002563 & 1.294\% \\
            3 & 0.066481 & 0.066427 & 0.081\% & 23 & 0.000739 & 0.000742 & 0.391\% \\
            4 & 0.066426 & 0.066530 & 0.156\% & 24 & 0.002952 & 0.002955 & 0.092\% \\
            5 & 0.018895 & 0.018834 & 0.321\% & 25 & 0.000756 & 0.000752 & 0.501\% \\
            6 & 0.050736 & 0.050770 & 0.068\% & 26 & 0.001830 & 0.001795 & 1.929\% \\
            7 & 0.008935 & 0.008883 & 0.579\% & 27 & 0.000821 & 0.000879 & 7.097\% \\
            8 & 0.016607 & 0.016649 & 0.255\% & 28 & 0.001587 & 0.001497 & 5.676\% \\
            9 & 0.007387 & 0.007456 & 0.937\% & 29 & 0.000461 & 0.000461 & 0.096\% \\
            10 & 0.009447 & 0.009511 & 0.674\% & 30 & 0.001680 & 0.001678 & 0.091\% \\
            11 & 0.003402 & 0.003362 & 1.188\% & 31 & 0.000402 & 0.000412 & 2.458\% \\
            12 & 0.011809 & 0.011720 & 0.755\% & 32 & 0.001038 & 0.001012 & 2.497\% \\
            13 & 0.002398 & 0.002453 & 2.279\% & 33 & 0.000605 & 0.000597 & 1.302\% \\
            14 & 0.006819 & 0.006839 & 0.299\% & 34 & 0.001049 & 0.001028 & 2.044\% \\
            15 & 0.003359 & 0.003316 & 1.281\% & 35 & 0.000451 & 0.000486 & 7.656\% \\
            16 & 0.004152 & 0.004080 & 1.726\% & 36 & 0.001312 & 0.001310 & 0.162\% \\
            17 & 0.001375 & 0.001390 & 1.081\% & 37 & 0.000281 & 0.000260 & 7.392\% \\
            18 & 0.005637 & 0.005641 & 0.066\% & 38 & 0.000835 & 0.000876 & 4.961\% \\
            19 & 0.001094 & 0.001096 & 0.219\% & 39 & 0.000426 & 0.000458 & 7.418\% \\
            20 & 0.007085 & 0.007095 & 0.134\% & 40 & 0.001771 & 0.001766 & 0.303\% \\
            \bottomrule \\
        \end{tabular}
        \caption{Comparison of $\delta_{\bm{u}}(t)$ and $\widetilde{\delta}_{\bm{u}}(t)$ for $a_1 = 1$ and $a_2 = 1$.}\label{tab:1-1}
    \end{table}
\end{center}

\begin{center}
    \begin{table}
        \begin{tabular}{cccc|cccc}
            \toprule
            $t$ & $\delta_{\bm{u}}(t)$ & $\widetilde{\delta}_{\bm{u}}(t)$ & error & $t$ & $\delta_{\bm{u}}(t)$ & $\widetilde{\delta}_{\bm{u}}(t)$ & error \\
            \midrule
            1 & 0.000000 & 0.000000 & 0.000\% & 21 & 0.000000 & 0.000000 & 0.000\% \\
            2 & 0.560934 & 0.561025 & 0.016\% & 22 & 0.005104 & 0.005157 & 1.045\% \\
            3 & 0.000000 & 0.000000 & 0.000\% & 23 & 0.000000 & 0.000000 & 0.000\% \\
            4 & 0.149582 & 0.149481 & 0.068\% & 24 & 0.012465 & 0.012304 & 1.293\% \\
            5 & 0.000000 & 0.000000 & 0.000\% & 25 & 0.000000 & 0.000000 & 0.000\% \\
            6 & 0.099722 & 0.099698 & 0.024\% & 26 & 0.003598 & 0.003597 & 0.014\% \\
            7 & 0.000000 & 0.000000 & 0.000\% & 27 & 0.000000 & 0.000000 & 0.000\% \\
            8 & 0.028047 & 0.028217 & 0.607\% & 28 & 0.003574 & 0.003516 & 1.620\% \\
            9 & 0.000000 & 0.000000 & 0.000\% & 29 & 0.000000 & 0.000000 & 0.000\% \\
            10 & 0.028342 & 0.028577 & 0.829\% & 30 & 0.005039 & 0.005052 & 0.267\% \\
            11 & 0.000000 & 0.000000 & 0.000\% & 31 & 0.000000 & 0.000000 & 0.000\% \\
            12 & 0.016620 & 0.016633 & 0.077\% & 32 & 0.001753 & 0.001770 & 0.974\% \\
            13 & 0.000000 & 0.000000 & 0.000\% & 33 & 0.000000 & 0.000000 & 0.000\% \\
            14 & 0.013402 & 0.013374 & 0.210\% & 34 & 0.002063 & 0.002128 & 3.166\% \\
            15 & 0.000000 & 0.000000 & 0.000\% & 35 & 0.000000 & 0.000000 & 0.000\% \\
            16 & 0.007012 & 0.007062 & 0.718\% & 36 & 0.001847 & 0.001930 & 4.511\% \\
            17 & 0.000000 & 0.000000 & 0.000\% & 37 & 0.000000 & 0.000000 & 0.000\% \\
            18 & 0.011080 & 0.011106 & 0.233\% & 38 & 0.001640 & 0.001595 & 2.768\% \\
            19 & 0.000000 & 0.000000 & 0.000\% & 39 & 0.000000 & 0.000000 & 0.000\% \\
            20 & 0.007558 & 0.007560 & 0.029\% & 40 & 0.001417 & 0.001418 & 0.064\% \\
            \bottomrule \\
        \end{tabular}
        \caption{Comparison of $\delta_{\bm{u}}(t)$ and $\widetilde{\delta}_{\bm{u}}(t)$ for $a_1 = 4$ and $a_2 = -1$.}\label{tab:4--1}
    \end{table}
\end{center}

\begin{center}
    \begin{table}
        \begin{tabular}{cccc|cccc}
            \toprule
            $t$ & $\delta_{\bm{u}}(t)$ & $\widetilde{\delta}_{\bm{u}}(t)$ & error & $t$ & $\delta_{\bm{u}}(t)$ & $\widetilde{\delta}_{\bm{u}}(t)$ & error \\
            \midrule
            1 & 0.224373 & 0.224381 & 0.003\% & 21 & 0.004765 & 0.004761 & 0.088\% \\
            2 & 0.168280 & 0.168315 & 0.021\% & 22 & 0.001531 & 0.001548 & 1.104\% \\
            3 & 0.199443 & 0.199323 & 0.060\% & 23 & 0.000443 & 0.000446 & 0.572\% \\
            4 & 0.056093 & 0.056196 & 0.183\% & 24 & 0.018698 & 0.018774 & 0.408\% \\
            5 & 0.011337 & 0.011407 & 0.620\% & 25 & 0.000453 & 0.000455 & 0.337\% \\
            6 & 0.149582 & 0.149463 & 0.080\% & 26 & 0.001079 & 0.001082 & 0.254\% \\
            7 & 0.005361 & 0.005354 & 0.128\% & 27 & 0.002462 & 0.002485 & 0.924\% \\
            8 & 0.000000 & 0.000000 & 0.000\% & 28 & 0.001340 & 0.001352 & 0.880\% \\
            9 & 0.022160 & 0.022184 & 0.107\% & 29 & 0.000276 & 0.000290 & 4.946\% \\
            10 & 0.008503 & 0.008521 & 0.217\% & 30 & 0.007558 & 0.007623 & 0.862\% \\
            11 & 0.002041 & 0.002023 & 0.904\% & 31 & 0.000241 & 0.000230 & 4.671\% \\
            12 & 0.024930 & 0.024918 & 0.050\% & 32 & 0.000000 & 0.000000 & 0.000\% \\
            13 & 0.001439 & 0.001428 & 0.765\% & 33 & 0.001815 & 0.001792 & 1.247\% \\
            14 & 0.004021 & 0.003973 & 1.185\% & 34 & 0.000619 & 0.000663 & 7.141\% \\
            15 & 0.010077 & 0.010214 & 1.358\% & 35 & 0.000271 & 0.000285 & 5.219\% \\
            16 & 0.000000 & 0.000000 & 0.000\% & 36 & 0.002770 & 0.002769 & 0.038\% \\
            17 & 0.000825 & 0.000860 & 4.232\% & 37 & 0.000168 & 0.000180 & 6.855\% \\
            18 & 0.016620 & 0.016679 & 0.353\% & 38 & 0.000492 & 0.000478 & 2.870\% \\
            19 & 0.000656 & 0.000627 & 4.445\% & 39 & 0.001279 & 0.001292 & 1.007\% \\
            20 & 0.002834 & 0.002852 & 0.628\% & 40 & 0.000000 & 0.000000 & 0.000\% \\
            \bottomrule \\
        \end{tabular}
        \caption{Comparison of $\delta_{\bm{u}}(t)$ and $\widetilde{\delta}_{\bm{u}}(t)$ for $a_1 = 10$ and $a_2 = 2$.}\label{tab:10-2}
    \end{table}
\end{center}

\FloatBarrier

\end{document}